\documentclass[12pt,reqno]{article}
\usepackage{amssymb,amscd,amsmath,amsthm,color}
\newtheorem{thm}{Theorem}[section]

\newtheorem{lemma}[thm]{Lemma}
\newtheorem{cor}[thm]{Corollary}

\newtheorem{remark}[thm]{Remark}
\newtheorem{example}[thm]{Example}
\newtheorem{problem}[thm]{Problem}
\numberwithin{equation}{section}

\def\diag{\mathrm{diag}}
\def\<{\langle}
\def\>{\rangle}
\def\bC{\mathbb{C}}
\def\lin{\mathrm{span}}
\def\Tr{\mathrm{Tr}\,}
\def\bM{\mathbb{M}}
\def\cM{\mathcal{M}}
\def\ran{\mathrm{ran}\,}
\def\bN{\mathbb{N}}
\def\cR{\mathcal{R}}
\def\bR{\mathbb{R}}
\def\eps{\varepsilon}
\def\ffi{\varphi}

\begin{document}
\baselineskip=16pt
\allowdisplaybreaks

\centerline{\LARGE Matrix limit theorems of Kato type}
\medskip
\centerline{\LARGE related to positive linear maps and operator means}
\bigskip
\bigskip
\centerline{\Large Fumio Hiai\footnote{{\it E-mail address:} hiai.fumio@gmail.com}}

\medskip
\begin{center}
$^1$\,Tohoku University (Emeritus), \\
Hakusan 3-8-16-303, Abiko 270-1154, Japan
\end{center}

\medskip
\begin{abstract}
We obtain limit theorems for $\Phi(A^p)^{1/p}$ and $(A^p\sigma B)^{1/p}$ as $p\to\infty$
for positive matrices $A,B$, where $\Phi$ is a positive linear map between matrix algebras
(in particular, $\Phi(A)=KAK^*$) and $\sigma$ is an operator mean (in particular, the
weighted geometric mean), which are considered as certain reciprocal Lie-Trotter formulas
and also a generalization of Kato's limit to the supremum $A\vee B$ with respect to the
spectral order.


\end{abstract}

\section{Introduction}

For any matrices $X$ and $Y$, the well-known Lie-Trotter formula is the convergence
$$
\lim_{n\to\infty}(e^{X/n}e^{Y/n})^n=e^{X+Y}.
$$
The symmetric form with a continuous parameter is also well-known for positive semidefinite
matrices $A,B\ge0$ as
\begin{equation}\label{F-1.1}
\lim_{p\searrow0}(A^{p/2}B^pA^{p/2})^{1/p}=P_0\exp(\log A\dot+\log B),
\end{equation}
where $P_0$ is the orthogonal projection onto the intersection of the supports of $A,B$ and
$\log A\dot+\log B$ is defined as $P_0(\log A)P_0+P_0(\log B)P_0$. When $\sigma$ is an
operator mean \cite{KA} corresponding to an operator monotone function $f$ on $(0,\infty)$
such that $\alpha:=f'(1)$ is in $(0,1)$, the operator mean version of the
Lie-Trotter formula is also known to hold as
\begin{equation}\label{F-1.2}
\lim_{p\searrow0}(A^p\sigma B^p)^{1/p}=P_0\exp((1-\alpha)\log A\dot+\alpha\log B)
\end{equation}
for matrices $A,B\ge0$. In particular, when $\sigma$ is the geometric mean $A\#B$
(introduced first in \cite{PW} and further discussed in \cite{KA}) corresponding to the
operator monotone function $f(x)=x^{1/2}$, \eqref{F-1.2} yields
\begin{equation}\label{F-1.3}
\lim_{p\searrow0}(A^p\#B^p)^{2/p}=P_0\exp(\log A\dot+\log B),
\end{equation}
which has the same right-hand side as \eqref{F-1.1}. Due to the Araki-Lieb-Thirring
inequality and the Ando-Hiai log-majorization \cite{Ar,AH}, it is worthwhile to note that 
$(A^{p/2}B^pA^{p/2})^{1/p}$ and $(A^p\#B^p)^{2/p}$ both tend to
$P_0\exp(\log A\dot+\log B)$ as $p\searrow0$, with the former decreasing and
the latter increasing in the log-majorization order (see \cite{AH} for details on
log-majorization for matrices).

In the previous paper \cite{AuHi2}, under the name ``reciprocal Lie-Trotter formula", we
considered the question complementary to \eqref{F-1.1} and \eqref{F-1.3}, that is, about
what happens to the limits of $(A^{p/2}B^pA^{p/2})^{1/p}$ and $(A^p\#B^p)^{2/p}$ as $p$
tends to $\infty$ instead of $0$. If $A$ and $B$ are commuting, then
$(A^{p/2}B^pA^{p/2})^{1/p}=(A^p\#B^p)^{2/p}=AB$, independently of $p>0$. However, if $A$
and $B$ are not commuting, then the question is rather complicated. Indeed, although we can
prove the existence of the limit $\lim_{p\to\infty}(A^{p/2}B^pA^{p/2})^{1/p}$, the
description of the limit has a rather complicated combinatorial nature. Moreover, it is
unknown so far whether the limit of $(A^p\#B^p)^{2/p}$ as $p\to\infty$ exists or not. In
the present paper, we consider a similar (but seemingly a bit simpler) question about what
happens to the limits of $(BA^pB)^{1/p}$ and $(A^p\#B)^{1/p}$ as $p$ tends to $\infty$, the
case where $B$ is fixed without the $p$-power, in certain more general settings.

The rest of the paper is organized as follows. In Section 2, we first prove the existence of
the limit of $(KA^pK^*)^{1/p}$ as $p\to\infty$ and give the description of the limit in terms
of the diagonalization (eigenvalues and eigenvectors) data of $A$ and the images of the
eigenvectors by $K$. We then extend the result to the limit of $\Phi(A^p)^{1/p}$ as
$p\to\infty$ for a positive linear map $\Phi$ between matrix algebras. For instance, this
limit is applied to the map $\Phi(A\oplus B):=(A+B)/2$ to reformulate Kato's limit theorem
$((A^p+B^p)/2)^{1/p}\to A\vee B$ in \cite{Ka}. Another application is given to find the limit
formula as $\alpha\searrow0$ of the sandwiched $\alpha$-R\'enyi divergence \cite{MDSFT,WWY},
a new relative entropy relevant to quantum information theory. In Section 3, we discuss
the limit behavior of $(A^p\sigma B)^{1/p}$ as $p\to\infty$ for operator means $\sigma$. To
do this, we may assume without loss of generality that $B$ is an orthogonal projection $E$.
Under a certain condition on $\sigma$, we prove that $(A^p\sigma E)^{1/p}$ is decreasing as
$1\le p\to\infty$, so that the limit as $p\to\infty$ exists. Furthermore, when $\sigma$ is
the the weighted geometric mean, we obtain an explicit description of the limit in terms of
$E$ and the spectral projections of $A$.

It is worth noting that a limit formula in the same vein as those in \cite{Ka} and
this paper was formerly given in \cite{ACS1,ACS2} for the spectral shorting operation.

\section{$\lim_{p\to\infty}\Phi(A^p)^{1/p}$ for positive linear maps $\Phi$}

For each $n\in\bN$ we write $\bM_n$ for the $n\times n$ complex matrix algebra and $\bM_n^+$
for the set of positive semidefinite matrices in $\bM_n$. When $A\in\bM_n$ is positive
definite, we write $A>0$. We denote by $\Tr$ the usual trace functional on $\bM_n$. For
$A\in\bM_n^+$, $\lambda_1(A)\ge\dots\ge\lambda_n(A)$ are the eigenvalues of $A$ in decreasing
order with multiplicities, and $\ran A$ is the range of $A$.

Let $A\in\bM_n^+$ be given, whose diagonalization is
\begin{align}\label{F-2.1}
A=V\diag(a_1,\dots,a_n)V^*=\sum_{i=1}^na_i|v_i\>\<v_i|
\end{align}
with the eigenvalues $a_1\ge\dots\ge a_n$ and a unitary matrix $V=[v_1\ \cdots\ v_n]$ so
that $Av_i=a_iv_i$ for $1\le i\le d$. Let $K\in\bM_n$ and assume that $K\ne0$ (our problem
below is trivial when $K=0$). Consider the sequence of vectors $Kv_1,\dots,Kv_n$ in $\bC^n$.
Let $1\le l_1<l_2<\dots<l_m$ be chosen so that if $l_{k-1}<i<l_k$ for $1\le k\le m$, then
$Kv_i$ is in $\lin\{Kv_{l_1},\dots,Kv_{l_{k-1}}\}$ (this means, in particular, $Kv_i=0$ if
$i<l_1$). Then $\{Kv_{l_1},\dots,Kv_{l_m}\}$ is a linearly independent subset of
$\{Kv_1,\dots,Kv_n\}$, so we perform the Gram-Schmidt orthogonalization to obtain an
orthonormal vectors $u_1,\dots,u_m$ from $Kv_{l_1},\dots,Kv_{l_m}$. In particular,
$u_1=Kv_{l_1}/\|Kv_{l_1}\|$. The next theorem is our first limit theorem. This implicitly
says that the right-hand side of \eqref{F-2.2} is independent of the expression of
\eqref{F-2.1} (note that $v_i$'s are not unique for degenerate eigenvalues $a_i$).

\begin{thm}\label{T-2.1}
We have
\begin{equation}\label{F-2.2}
\lim_{p\to\infty}(KA^pK^*)^{1/p}=\sum_{k=1}^ma_{l_k}|u_k\>\<u_k|,
\end{equation}
and in particular,
\begin{equation}\label{F-2.3}
\lim_{p\to\infty}\lambda_k((KA^pK^*)^{1/p})
=\begin{cases}a_{l_k}, & 1\le k\le m, \\
0, & m<k\le n.\end{cases}
\end{equation}
\end{thm}

\begin{proof}
Write $Z_p:=(KA^pK^*)^{1/p}$ and $\lambda_i(p):=\lambda_i(Z_p)$ for $p>0$ and $1\le i\le n$.
First we prove \eqref{F-2.3}. Note that
\begin{align}\label{F-2.4}
Z_p^p&=KA^pK^*=KV\diag(a_1^p,\dots,a_n^p)V^*K^* \nonumber\\
&=\bigl[a_1^pKv_1\ a_2^pKv_2\ \cdots\ a_n^pKv_n\bigr]
\bigl[Kv_1\ Kv_2\ \cdots\ Kv_n\bigr]^*.
\end{align}
Since
\begin{align*}
\lambda_1(p)^p&\le\Tr Z_p^p
=\Tr[Kv_1\ \cdots\ Kv_n\bigr]^*\bigl[a_1^pKv_1\ \cdots\ a_n^pKv_n\bigr] \\
&=\sum_{i=1}^na_i^p\<Kv_i,Kv_i\>\le a_{l_1}^p\sum_{i=1}^n\|Kv_i\|^2,
\end{align*}
we have
$$
\limsup_{p\to\infty}\lambda_1(p)\le a_{l_1}.
$$
Moreover, since
$$
n\lambda_1(p)^p\ge\Tr Z_p^p=\sum_{i=1}^na_i^p\<Kv_i,Kv_i\>
\ge a_{l_1}^p\|Kv_{l_1}\|^2,
$$
we have
$$
\liminf_{p\to\infty}\lambda_1(p)\ge a_{l_1}.
$$
Therefore, \eqref{F-2.3} holds for $k=1$.

To prove \eqref{F-2.3} for $k>1$, we consider the \emph{antisymmetric tensor powers}
$A^{\wedge k}$ and $K^{\wedge k}$ for each $k=1,\dots,n$. Note that
\begin{align}\label{F-2.5}
Z_p^{\wedge k}=\bigl((K^{\wedge k})(A^{\wedge k})^p(K^{\wedge k})^*\bigr)^{1/p}
\end{align}
and
$$
A^{\wedge k}=\sum_{1\le i_1<\dots<i_k\le d}a_{i_1}\cdots a_{i_k}
|v_{i_1}\wedge\dots\wedge v_{i_k}\>\<v_{i_1}\wedge\dots\wedge v_{i_k}|.
$$
The above case applied to $A^{\wedge k}$ and $K^{\wedge k}$ yields that
\begin{align}\label{F-2.6}
&\lim_{p\to\infty}\lambda_1(p)\lambda_2(p)\cdots\lambda_k(p)
=\lim_{p\to\infty}\lambda_1\bigl(Z_p^{\wedge k}\bigr) \nonumber\\
&\quad=\max\bigl\{a_{i_1}\cdots a_{i_k}:1\le i_1<\dots<i_k\le d,\,
K^{\wedge k}(v_{i_1}\wedge\dots\wedge v_{i_k})\ne0\bigr\}.
\end{align}
Since $K^{\wedge k}(v_{i_1}\wedge\dots\wedge v_{i_k})=Kv_{i_1}\wedge\cdots\wedge Kv_{i_k}$
is non-zero if and only if $\{Kv_{i_1},\dots,Kv_{i_k}\}$ is linearly independent, it is
easy to see that \eqref{F-2.6} is equal to $a_{l_1}\cdots a_{l_k}$ if $k\le m$ and equal to
$0$ if $k>m$. Therefore,
$$
\lim_{p\to\infty}\lambda_1(p)\lambda_2(p)\cdots\lambda_k(p)
=\begin{cases}
a_{l_1}\cdots a_{l_k}, & 1\le k\le m, \\
0, & m<k\le n,\end{cases}
$$
which implies \eqref{F-2.3}.

Now, for $p>0$ choose an orthonormal basis $\{u_1(p),\dots,u_n(p)\}$ of $\bC^n$ for which
$Z_pu_i(p)=\lambda_i(p)u_i(p)$ for $1\le i\le n$. To prove \eqref{F-2.2}, write
$\tilde a_k:=a_{l_k}$ for $1\le k\le m$. If $\tilde a_1=0$ then it is obvious that
$\lim_{p\to\infty}Z_p=0$. So assume that $\tilde a_1>0$ and furthermore
$\tilde a_1>\tilde a_2$, i.e., $\lim_{p\to\infty}\lambda_1(p)>\lim_{p\to\infty}\lambda_2(p)$
at the moment. From \eqref{F-2.4} we have
\begin{align*}
Z_p^p&=\sum_{i=1}^na_i^p|Kv_i\>\<Kv_i| \\
&=\sum_{i=l_1}^{l_2-1}a_i^p|Kv_i\>\<Kv_i|
+\sum_{i=l_2}^na_i^p|Kv_i\>\<Kv_i| \\
&=\sum_{i=l_1}^{l_2-1}a_i^p\|Kv_i\|^2|u_1\>\<u_1|
+\sum_{i=l_2}^na_i^p|Kv_i\>\<Kv_i|
\end{align*}
so that
$$
\biggl({Z_p\over\tilde a_1}\biggr)^p
=\sum_{i=l_1}^{l_2-1}\biggl({a_i\over\tilde a_1}\biggr)^p\|Kv_i\|^2|u_1\>\<u_1|
+\sum_{i=l_2}^n\biggl({a_i\over\tilde a_1}\biggr)^p|Kv_i\>\<Kv_i|
\longrightarrow\alpha|u_1\>\<u_1|
$$
as $p\to\infty$ for some $\alpha>0$, since $a_i/\tilde a_1\le\tilde a_2/\tilde a_1<1$ for
$i\ge l_2$. Hence, for any $p>0$ sufficiently large, the largest eigenvalue of
$(Z_p/\tilde a_1)^p$ is simple and the corresponding eigen projection converges to
$|u_1\>\<u_1|$ as $p\to\infty$. Since the eigen projection $E_1(p)$ of $Z_p$ corresponding
to the largest eigenvalue $\lambda_1(p)$ (simple for any large $p>0$) is the same as that
of $(Z_p/\tilde a_1)^p$, we have
$$
E_1(p)=|u_1(p)\>\<u_1(p)|\longrightarrow|u_1\>\<u_1|
\quad\mbox{as $p\to\infty$}.
$$

In the general situation, we assume that $\tilde a_1\ge\dots\ge\tilde a_k>\tilde a_{k+1}$
with $1\le k\le m$, where $\tilde a_{k+1}=0$ if $k=m$. From \eqref{F-2.3} note that
$$
\lim_{p\to\infty}\lambda_1(Z_p^{\wedge k})
=\tilde a_1\cdots\tilde a_{k-1}\tilde a_k
>\tilde a_1\cdots\tilde a_{k-1}\tilde a_{k+1}
=\lim_{p\to\infty}\lambda_2(Z_p^{\wedge k}).
$$
Hence, for any sufficiently large $p>0$, the largest eigenvalue
$\lambda_1(Z_p^{\wedge k})=\lambda_1(p)\cdots\lambda_k(p)$ of $Z_p^{\wedge k}$ is simple,
and from the above case applied to \eqref{F-2.5} it follows that
$$
|u_1(p)\wedge\cdots\wedge u_k(p)\>\<u_1(p)\wedge\cdots\wedge u_k(p)|
\longrightarrow|u_1\wedge\cdots\wedge u_k\>\<u_1\wedge\cdots\wedge u_k|
\quad\mbox{as $p\to\infty$},
$$
since the vector in the present situation corresponding to $u_1=Kv_{l_1}/\|Kv_{l_1}\|$ is
$$
{K^{\wedge k}(v_{l_1}\wedge\cdots\wedge v_{l_k})\over
\|K^{\wedge k}(v_{l_1}\wedge\cdots\wedge v_{l_k})\|}
={Kv_{l_1}\wedge\cdots\wedge Kv_{l_k}\over\|Kv_{l_1}\wedge\cdots\wedge Kv_{l_k}\|}
=u_1\wedge\cdots\wedge u_k.
$$
(The last identity follows from the fact that, for linearly independent $w_1,\dots,w_k$,
$w_1\wedge\cdots\wedge w_k/\|w_1\wedge\cdots\wedge w_k\|=w_1'\wedge\cdots\wedge w_k'$ if
$w_1',\dots,w_k'$ are the Gram-Schmidt orthogonalization of $w_1,\dots,w_k$.) By
\cite[Lemma 2.4]{AuHi2} we see that the orthogonal projection $E_k(p)$ onto
$\lin\{u_1(p),\dots,u_k(p)\}$ converges to the orthogonal projection $E_k$ of
$\lin\{u_1,\dots,u_k\}$.

Finally, let $0=k_0<k_1<\dots<k_{s-1}<k_s=m$ be such that
$$
\tilde a_1=\dots=\tilde a_{k_1}>\tilde a_{k_1+1}=\dots=\tilde a_{k_2}
>\dots>\tilde a_{k_{s-1}+1}=\dots=\tilde a_{k_s}.
$$
The above argument says that, for every $r=1,\dots,s-1$, the orthogonal projection
$E_{k_r}(p)$ onto $\lin\{u_1(p),\dots,u_{k_r}(p)\}$ converges to the orthogonal projection
$E_{k_r}$ onto $\lin\{u_1,\dots,u_{k_r}\}$. When $\tilde a_{k_s}>0$, this holds for $r=s$ as
well. Therefore, when $\tilde a_{k_s}>0$, we have
\begin{align*}
Z_p&=\sum_{i=1}^n\lambda_i(p)|u_i(p)\>\<u_i(p)| \\
&=\sum_{r=1}^s\sum_{i=k_{r-1}+1}^{k_r}\lambda_i(p)|u_i(p)\>\<u_i(p)|
+\sum_{i=k_s+1}^n\lambda_i(p)|u_i(p)\>\<u_i(p)| \\
&=\sum_{r=1}^s\sum_{i=k_{r-1}+1}^{k_r}(\lambda_i(p)-\tilde a_i)|u_i(p)\>\<u_i(p)|
+\sum_{r=1}^s\tilde a_{k_r}(E_{k_r}(p)-E_{k_{r-1}}(p)) \\
&\qquad+\sum_{i=k_s+1}^n\lambda_i(p)|u_i(p)\>\<u_i(p)| \\
&\longrightarrow\sum_{r=1}^s\tilde a_{k_r}(E_{k_r}-E_{k_{r-1}})
=\sum_{i=1}^m\tilde a_i|u_i\>\<u_i|,\quad\mbox{where $E_0(p)=E_0=0$}.
\end{align*}
When $\tilde a_{k_s}=0$, we may modify the above estimate as
\begin{align*}
Z_p&=\sum_{r=1}^{s-1}\sum_{i=k_{r-1}+1}^{k_r}\lambda_i(p)|u_i(p)\>\<u_i(p)|
+\sum_{i=k_{s-1}+1}^n\lambda_i(p)|u_i(p)\>\<u_i(p)| \\
&\longrightarrow\sum_{r=1}^{s-1}\tilde a_{k_r}(E_{k_r}-E_{k_{r-1}})
=\sum_{i=1}^{k_{s-1}}\tilde a_i|u_i\>\<u_i|=\sum_{i=1}^m\tilde a_i|u_i\>\<u_i|.
\end{align*}
\end{proof}

The following corollary of Theorem \ref{T-2.1} is an improvement of \cite[Theorem 1.2]{Bo}.

\begin{cor}\label{C-2.2}
Let $A\in\bM_n$ be positive definite. We have
$\lim_{p\to\infty}\lambda_i((KA^pK^*)^{1/p})=a_i$ for all $i=1,\dots,n$ if and only if
$\{Kv_1,\dots,Kv_n\}$ is linearly independent.
\end{cor}

\begin{remark}\label{R-2.3}\rm
Note that Theorem \ref{T-2.1} can easily extend to the case where $K$ is a rectangle
$n'\times n$ matrix. In fact, when $n'<n$ we may apply Theorem \ref{T-2.1} to $n\times n$
matrices $\begin{bmatrix}K\\O\end{bmatrix}$ and $A$, and when $n'>n$ we may apply to
$n'\times n'$ matrices $\begin{bmatrix}K&O\end{bmatrix}$ and $A\oplus O_{n'-n}$.
\end{remark}

A linear map $\Phi:\bM_n\to\bM_{n'}$ is said to be positive if $\Phi(A)\in\bM_{n'}^+$ for
all $A\in\bM_n^+$, which is further said to be strictly positive if $\Phi(I_n)>0$, that is,
$\Phi(A)>0$ for all $A\in\bM_n$, $A>0$. The following is an extended and refined version of
Theorem \ref{T-2.1}.

\begin{thm}\label{T-2.4}
Let $\Phi:\bM_n\to\bM_{n'}$ be a positive linear map. Let $A\in\bM_n^+$ be given as
$A=\sum_{i=1}^na_i|v_i\>\<v_i|$ with $a_1\ge\dots\ge a_n$ and an orthonormal basis
$\{v_1,\dots,v_n\}$ of $\bC^n$. Then $\lim_{p\to\infty}\Phi(A^p)^{1/p}$ exists and
$$
\lim_{p\to\infty}\Phi(A^p)^{1/p}=\sum_{i=1}^na_iP_{\cM_i},
$$
where
\begin{align*}
\cM_1&:=\ran\Phi(|v_1\>\<v_1|), \\
\cM_i&:=\bigvee_{j=1}^i\ran\Phi(|v_j\>\<v_j|)\ominus
\bigvee_{j=1}^{i-1}\ran\Phi(|v_j\>\<v_j|),\qquad2\le i\le n,
\end{align*}
and $P_{\cM_i}$ is the orthogonal projection onto $\cM_i$ for $1\le i\le n$.
\end{thm}

\begin{proof}
Let $C^*(I,A)$ be the commutative $C^*$-subalgebra of $\bM_n$ generated by $I,A$.
We can consider the composition of the conditional expectation from $\bM_n$ onto $C^*(I,A)$
with respect to $\Tr$ and $\Phi|_{C^*(I,A)}:C^*(I,A)\to\bM_{n'}$ instead of $\Phi$, so we
may assume that $\Phi$ is completely positive. By the Stinespring representation there are
a $\nu\in\bN$, a $*$-homomorphism $\pi:\bM_n\to\bM_{n\nu}$ and a linear map
$K:\bC^{n\nu}\to\bC^{n'}$ such that $\Phi(X)=K\pi(X)K^*$ for all $X\in\bM_n$. Moreover,
since $\pi:\bM_n\to\bM_{n\nu}$ is represented, under a suitable change of an orthonormal
basis of $\bC^{n\nu}$, as $\Phi(X)=I_\nu\otimes X$ for all $X\in\bM_n$ under
identification $\bM_{n\nu}=\bM_\nu\otimes\bM_n$, we can assume that $\Phi$ is given (with
a change of $K$) as
$$
\Phi(X)=K(I_\nu\otimes X)K^*,\qquad X\in\bM_n.
$$
We then write
\begin{align*}
I_\nu\otimes A&=(I_\nu\otimes V)\diag\bigl(\underbrace{a_1,\dots,a_1}_\nu,
\underbrace{a_2,\dots,a_2}_\nu,\dots,\underbrace{a_n,\dots,a_n}_\nu\bigr)
(I_\nu\otimes V)^* \\
&=\sum_{i=1}^na_i(|e_1\otimes v_i\>\<e_1\otimes v_i|+\dots
+|e_\nu\otimes v_i\>\<e_\nu\otimes v_i|).
\end{align*}
Now, we consider the following sequence of $n\nu$ vectors in $\bC^{n'}$:
$$
K(e_1\otimes v_1),\dots,K(e_\nu\otimes v_1),K(e_1\otimes v_2),\dots,K(e_\nu\otimes v_2),
\dots,K(e_1\otimes v_n),\dots,K(e_\nu\otimes v_n),
$$
and if $K(e_j\otimes v_i)$ is a linear combination of the vectors in the sequence
preceding it, then we remove it from the sequence. We write the resulting linearly
independent subsequence as
$$
K(e_j\otimes v_{l_1})\ \,(j\in J_1),\ \ K(e_j\otimes v_{l_2})\ \,(j\in J_2),
\ \ \dots,\ \ K(e_j\otimes v_{l_m})\ \,(j\in J_m),
$$
where $1\le l_1<l_2<\dots<l_m\le n$ and $J_1,\dots,J_m\subset\{1,\dots,\nu\}$.
Furthermore, by performing the Gram-Schmidt orthogonalization to this subsequence, we
end up making an orthonormal sequence of vectors in $\bC^{n'}$ as follows:
$$
u_j^{(l_1)}\ \,(j\in J_1),\ \ u_j^{(l_2)}\ \,(j\in J_2),
\ \ \dots,\ \ u_j^{(l_m)}\ \,(j\in J_m).
$$
Since
$$
\Phi(A^p)^{1/p}=K((I_\nu\otimes A)^p)^{1/p}K^*,
$$
Theorem \ref{T-2.1} and Remark \ref{R-2.3} imply that $\lim_{p\to\infty}\Phi(A^p)^{1/p}$
exists and
$$
\lim_{p\to\infty}\Phi(A^p)^{1/p}
=\sum_{k=1}^ma_{l_k}\Biggl(\sum_{j\in J_k}
\big|u_j^{(l_k)}\bigr\>\bigl\<u_j^{(l_k)}\big|\Biggr),
$$
where $\sum_{j\in J_k}\big|u_j^{(l_k)}\bigr\>\bigl\<u_j^{(l_k)}\big|=0$ if $J_k=\emptyset$.

The next step of the proof is to find what is
$\sum_{j\in J_k}\big|u_j^{(l_k)}\bigr\>\bigl\<u_j^{(l_k)}\big|$ for $1\le k\le m$. For
this we first note that
\begin{align*}
\sum_{j=1}^\nu|K(e_j\otimes v_i)\>\<K(e_j\otimes v_i)|
&=K\Biggl(\sum_{j=1}^\nu|e_j\otimes v_i\>\<e_j\otimes v_i|\Biggr)K^* \\
&=K(I_\nu\otimes|v_i\>\<v_i|)K^*=\Phi(|v_i\>\<v_i|).
\end{align*}
From Lemma \ref{L-2.5} below this implies that
$$
\cR_i:=\ran\Phi(|v_i\>\<v_i|)=\lin\{K(e_j\otimes v_i):1\le j\le\nu\}.
$$
Through the procedure of the Gram-Schmidt diagonalization we see that
\begin{align*}
&\cR_i=0,\qquad1\le i<l_1, \\
&\cR_{l_1}=\lin\bigl\{u_j^{(l_1)}:j\in J_1\bigr\}, \\
&\cR_i\subset\cR_{l_1},\qquad l_1<i<l_2, \\
&(\cR_{l_1}\vee\cR_{l_2})\ominus\cR_{l_1}=\lin\bigl\{u_j^{(l_2)}:j\in J_2\bigr\}, \\
&\cR_i\subset\cR_{l_1}\vee\cR_{l_2},\qquad l_2<i<l_3, \\
&(\cR_{l_1}\vee\cR_{l_2}\vee\cR_{l_3})\ominus(\cR_{l_1}\vee\cR_{l_2})
=\lin\bigl\{u_j^{(l_3)}:j\in J_3\bigr\}, \\
&\qquad\vdots \\
&\cR_i\subset\cR_{l_1}\vee\dots\vee\cR_{l_{m-1}},\qquad l_{m-1}<i<l_m, \\
&(\cR_{l_1}\vee\dots\vee\cR_{l_m})\ominus(\cR_{l_1}\vee\dots\vee\cR_{l_{m-1}})
=\lin\bigl\{u_j^{(l_m)}:j\in J_m\bigr\}, \\
&\cR_i\subset\cR_{l_1}\vee\dots\vee\cR_{l_m},\qquad l_m<i\le n.
\end{align*}
Now, let $P_{\cM_i}$ be the orthogonal projections, respectively, onto the subspaces
$$
\cM_1:=\cR_1,\quad
\cM_i:=(\cR_1\vee\dots\vee\cR_i)\ominus(\cR_1\vee\dots\vee\cR_{i-1}),\quad2\le i\le n,
$$
so that $P_{\cM_i}=0$ if $i\not\in\{l_1,\dots,l_m\}$ and $P_{\cM_{l_k}}$ is the orthogonal
projection onto $\lin\bigl\{u_j^{(l_k)}:j\in J_k\bigr\}$ for $1\le k\le m$. Therefore,
we have
$$
\lim_{p\to\infty}\Phi(A^p)^{1/p}
=\sum_{k=1}^ma_{l_k}
\Biggl(\sum_{j\in J_k}\big|u_j^{(l_k)}\bigr\>\bigl\<u_j^{(l_k)}\big|\Biggr)
=\sum_{k=1}^ma_{l_k}P_{\cM_{l_k}}=\sum_{i=1}^na_iP_{\cM_i}.
$$
\end{proof}

\begin{lemma}\label{L-2.5}
For any finite set $\{w_1,\dots,w_k\}$ in $\bC^{n'}$, $\lin\{w_1,\dots,w_k\}$ is equal to
the range of $|w_1\>\<w_1|+\dots+|w_k\>\<w_k|$. More generally, for every
$B_1,\dots,B_k\in\bM_{n'}^+$, $\bigvee_{j=1}^k\ran B_j$ is equal to the range of
$B_1+\dots+B_k$.
\end{lemma}

\begin{proof}
Let $Q:=|w_1\>\<w_1|+\dots+|w_k\>\<w_k|$. Since
$$
Qx=\<w_1,x\>w_1+\dots+\<w_k,x\>w_k\in\lin\{w_1,\dots,w_k\}
$$
for all $x\in\bC^{n'}$, we have $\ran Q\subset\lin\{w_1,\dots,w_k\}$. Since
$|w_i\>\<w_i|\le Q$, we have
$$
w_i\in\ran|w_i\>\<w_i|\subset\ran Q,\quad1\le i\le k.
$$
Hence we have $\lin\{w_1,\dots,w_k\}\subset\ran Q$. The proof of the latter assertion is
similar.
\end{proof}

Thanks to the lemma we can restate Theorem \ref{T-2.4} as follows:

\begin{thm}\label{T-2.6}
Let $\Phi:\bM_n\to\bM_{n'}$ be a positive linear map. Let $A\in\bM_n^+$ be given with the
spectral decomposition $A=\sum_{k=1}^ma_kP_k$, where $a_1>a_2>\dots>a_m>0$. Define
\begin{align*}
\cM_1&:=\ran\Phi(P_1), \\
\cM_k&:=\ran\Phi(P_1+\dots+P_k)\ominus\ran\Phi(P_1+\dots+P_{k-1}),\quad2\le k\le m.
\end{align*}
Then
$$
\lim_{p\to\infty}\Phi(A^p)^{1/p}=\sum_{k=1}^ma_kP_{\cM_k}.
$$
\end{thm}

\begin{example}\label{E-2.7}\rm
Consider a linear map $\Phi:\bM_{2n}\to\bM_n$ given by
$$
\Phi\left(\begin{bmatrix}X_{11}&X_{12}\\X_{21}&X_{22}\end{bmatrix}\right)
:={X_{11}+X_{22}\over2},\qquad X_{ij}\in\bM_n.
$$
Clearly, $\Phi$ is completely positive. For any $A,B\in\bM_n^+$ and $p>0$ we have
$$
\Phi\left(\begin{bmatrix}A&0\\0&B\end{bmatrix}^p\right)^{1/p}
=\biggl({A^p+B^p\over2}\biggr)^{1/p}.
$$
Thus, it is well-known \cite{Ka} that
\begin{align}\label{F-2.7}
\lim_{p\to\infty}\Phi\left(\begin{bmatrix}A&0\\0&B\end{bmatrix}^p\right)^{1/p}
=\lim_{p\to\infty}\biggl({A^p+B^p\over2}\biggr)^{1/p}=A\vee B,
\end{align}
where $A\vee B$ is the supremum of $A,B$ in the spectral order. Here let us show
\eqref{F-2.7} from Theorem \ref{T-2.6}. The spectral decompositions of $A,B$ are given as
$$
A=\sum_{i=1}^ma_iP_i,\qquad B=\sum_{j=1}^{m'}b_jQ_j,
$$
where $a_1>\dots>a_m\ge0$, $b_1>\dots>b_{m'}\ge0$ and
$\sum_{i=1}^mP_i=\sum_{j=1}^{m'}Q_j=I$. Then
$$
A\oplus B=\sum_{k=1}^lc_kR_k,
$$
where $\{c_k\}_{k=1}^l=\{a_i\}_{i=1}^m\cup\{b_j\}_{j=1}^{m'}$ with $c_1>\dots>c_l$ and
$$
R_k=\begin{cases}P_i\oplus Q_j & \text{if $a_i=b_j=c_k$}, \\
P_i\oplus0 & \text{if $a_i=c_k$ and $b_j\ne c_k$ for all $j$}, \\
0\oplus Q_j & \text{if $b_j=c_k$ and $a_i\ne c_k$ for all $i$}.
\end{cases}
$$
Note that
$$
\Phi(R_1+\dots+R_k)
={1\over2}\Biggl(\sum_{i:a_i\ge c_k}P_i+\sum_{j:b_j\ge c_k}Q_j\Biggr)
$$
so that by Lemma \ref{L-2.5} the support projection $F_k$ (i.e., the orthogonal projection
onto the range) of $\Phi(R_1+\dots+R_k)$ is
$$
F_k=\Biggl(\sum_{i:a_i\ge c_k}P_i\Biggr)\vee\Biggl(\sum_{j:b_j\ge c_k}Q_j\Biggr).
$$
Theorem \ref{T-2.6} implies that
$$
\lim_{p\to\infty}\Phi\left(\begin{bmatrix}A&0\\0&B\end{bmatrix}^p\right)^{1/p}
=C:=\sum_{k=1}^lc_k(F_k-F_{k-1}).
$$
For every $x\in\bR$ we denote by $E_{[x,\infty)}(A)$ the spectral projection of $A$
corresponding to the interval $[x,\infty)$, i.e.,
$$
E_{[x,\infty)}(A):=\sum_{i:a_i\ge x}P_i,
$$
and similarly for $E_{[x,\infty)}(B)$ and $E_{[x,\infty)}(C)$. If $c_k\ge x>c_{k+1}$
for some $1\le k<l$, then we have
$$
E_{[x,\infty)}(C)=F_k=E_{[x,\infty)}(A)\vee E_{[x,\infty)}(B).
$$
This holds also when $x>c_1$ and $x\le c_l$. Indeed, when $x>c_1$,
$E_{[x,\infty)}(C)=0=E_{[x,\infty)}(A)\vee E_{[x,\infty)}(B)$. When $x\le c_l$,
$E_{[x,\infty)}(C)=I=E_{[x,\infty)}(A)\vee E_{[x,\infty)}(B)$. This description of $C$ is the
same as $A\vee B$ in \cite{Ka}, so we have $C=A\vee B$.
\end{example}

\begin{example}\label{E-2.8}\rm
The example here is relevant to quantum information. For density matrices $\rho,\sigma\in\bM_n$
(i.e., $\rho,\sigma\in\bM_n^+$ with $\Tr\rho=\Tr\sigma=1$) and for a parameter
$\alpha\in(0,\infty)\setminus\{1\}$, the traditional \emph{R\'enyi relative entropy} is
$$
D_\alpha(\rho\|\sigma):=\begin{cases}
{1\over\alpha-1}\log\bigl[\Tr\rho^\alpha\sigma^{1-\alpha}\bigr]
& \text{if $\rho^0\le\sigma^0$ or $0<\alpha<1$}, \\
+\infty & \text{otherwise},
\end{cases}
$$
where $\rho^0$ denotes the support projection of $\rho$. On the other hand, the new concept
recently introduced and called the \emph{sandwiched R\'enyi relative entropy}
\cite{MDSFT,WWY} is
$$
\widetilde D_\alpha(\rho\|\sigma):=\begin{cases}
{1\over\alpha-1}\log\bigl[\Tr
\bigl(\sigma^{1-\alpha\over2\alpha}\rho\sigma^{1-\alpha\over2\alpha}\bigr)^\alpha\bigr]
& \text{if $\rho^0\le\sigma^0$ or $0<\alpha<1$}, \\
+\infty & \text{otherwise}.
\end{cases}
$$
By taking the limit we also consider
\begin{align*}
D_0(\rho\|\sigma)&:=\lim_{\alpha\searrow0}D_\alpha(\rho\|\sigma)=-\log\Tr(\rho^0\sigma), \\
\widetilde D_0(\rho\|\sigma)&:=\lim_{\alpha\searrow0}\widetilde D_\alpha(\rho\|\sigma)
=-\log\biggl[\lim_{\alpha\searrow0}\Tr
\bigl(\sigma^{1-\alpha\over2\alpha}\rho\sigma^{1-\alpha\over2\alpha}\bigr)^\alpha\biggr].
\end{align*}
(We remark that the notations $D_\alpha$ and $\widetilde D_\alpha$ are interchanged from
those in \cite{DL}.) Here, note that
$$
\lim_{\alpha\searrow0}\Tr
\bigl(\sigma^{1-\alpha\over2\alpha}\rho\sigma^{1-\alpha\over2\alpha}\bigr)^\alpha
=\lim_{p\to\infty}\Tr(\sigma^{p/2}\rho\sigma^{p/2})^{1/p}
=\lim_{p\to\infty}\Tr(\rho^0\sigma^p\rho^0)^{1/p},
$$
where the existence of $\lim_{p\to\infty}\Tr(\rho^0\sigma^p\rho^0)^{1/p}$ follows from the
\emph{Araki-Lieb-Thirring inequality} \cite{Ar} (also \cite{AH}), and the latter equality
above follows since $\lambda\rho^0\le\rho\le\mu\rho^0$ for some $\lambda,\mu>0$ and
$$
\lambda^{1/p}\Tr(\sigma^{p/2}\rho^0\sigma^{p/2})^{1/p}
\le\Tr(\sigma^{p/2}\rho\sigma^{p/2})^{1/p}
\le\mu^{1/p}\Tr(\sigma^{p/2}\rho^0\sigma^{p/2})^{1/p}.
$$
It was proved in \cite{DL} that
\begin{align}\label{F-2.8}
\widetilde D_0(\rho\|\sigma)\le D_0(\rho\|\sigma)
\end{align}
and equality holds in \eqref{F-2.8} if $\rho^0=\sigma^0$. Let us here prove the following:
\begin{itemize}
\item[(1)] $\widetilde D_0(\rho\|\sigma)=-\log\widetilde Q_0(\rho\|\sigma)$, where
\begin{align*}
\widetilde Q_0(\rho\|\sigma):=\max\bigl\{\Tr(P\sigma)&:
P\ \mbox{an orthogonal projection}, \\
&\qquad\qquad[P,\sigma]=0,\,(P\rho^0P)^0=P\bigr\}.
\end{align*}
\item[(2)] $\widetilde D_0(\rho\|\sigma)=D_0(\rho\|\sigma)$ holds if and only if
$[\rho^0,\sigma]=0$. (Obviously, $[\rho^0,\sigma]=0$ if $\rho^0=\sigma^0$.)
\end{itemize}

Indeed, to prove (1), first note that $(P\rho^0P)^0=P$ means that the dimension of
$\ran\rho^0P$ is equal to that of $P$, that is, $\rho^0v_1,\dots,\rho^0v_d$ are linearly
independent when $\{v_1,\dots,v_d\}$ is an orthonormal basis of $\ran P$. Choose
$1\le l_1<l_2<\dots<l_m$ as in the first paragraph of this section (before Theorem
\ref{T-2.1}) for $A=\sigma$ and $K=\rho^0$. Let $P_0$ be the orthogonal projection onto
$\lin\{v_{l_1},\dots,v_{l_m}\}$. Then $[P_0,\sigma]=0$, $(P_0\rho^0P_0)^0=P_0$, and Theorem
\ref{T-2.1} gives
$$
\lim_{p\to\infty}\Tr(\rho^0\sigma^p\rho^0)^{1/p}
=\sum_{k=1}^ma_{l_k}=\Tr(P_0\sigma).
$$
On the other hand, let $P$ be an orthogonal projection with $[P,\sigma]=0$ and
$(P\rho^0P)^0=P$. From $[P,\sigma]=0$ we may assume that $P=\sum_{k=1}^d|v_{i_k}\>\<v_{i_k}|$
for some $1\le i_1<\dots<i_d\le n$ (after, if necessary, changing $v_i$ for degenerate
eigenvalues $a_i$). Since $(P\rho^0P)^0=P$ implies that $\rho^0v_{i_1},\dots,\rho^0v_{i_d}$
are linearly independent, we have $d\le m$ and
$$
\Tr(P\sigma)=\sum_{k=1}^da_{i_k}\le\sum_{k=1}^ma_{l_k}=\Tr(P_0\sigma).
$$
Next, to prove (2), note that $\Tr(\rho^0\sigma^p\rho^0)^{1/p}$ is increasing in $p>0$ by
the Araki-Lieb-Thirring inequality mentioned above, which shows that
$$
\Tr(\rho^0\sigma)\le\lim_{p\to\infty}\Tr(\rho^0\sigma^p\rho^0)^{1/p}.
$$
This means inequality \eqref{F-2.8}, and equality holds in \eqref{F-2.8} if and only if
$\Tr(\rho^0\sigma^p\rho^0)^{1/p}$ is constant for $p\ge1$. By \cite[Theorem 2.1]{Hi} this
is equivalent to the commutativity $\rho^0\sigma=\sigma\rho^0$.
\end{example}

Finally, we consider the complementary convergence of $\Phi(A^p)^{1/p}$ as $p\to-\infty$,
or $\Phi(A^{-p})^{-1/p}$ as $p\to\infty$. Here, the expression $\Phi(A^{-p})^{-1/p}$ for
$p>0$ is defined in such a way that the $(-p)$-power of $A$ is restricted to the support
of $A$, i.e., defined in the sense of the generalized inverse, and the $(-1/p)$-power of
$\Phi(A^{-p})$ is also in this sense. The next theorem is the complementary counterpart of
Theorem \ref{T-2.6}.

\begin{thm}\label{T-2.9}
Let $\Phi:\bM_n\to\bM_{n'}$ be a positive linear map. Let $A\in\bM_n^+$ be given with the
spectral decomposition $A=\sum_{k=1}^ma_kP_k$, where $a_1>a_2>\dots>a_m>0$. Define
\begin{align*}
\widetilde\cM_k&:=\ran\Phi(P_k+\dots+P_m)\ominus\ran\Phi(P_{k+1}+\dots+P_m),
\quad1\le i\le m-1, \\
\widetilde\cM_m&:=\ran\Phi(P_m).
\end{align*}
Then
$$
\lim_{p\to\infty}\Phi(A^{-p})^{-1/p}=\sum_{k=1}^ma_kP_{\widetilde\cM_k}.
$$
\end{thm}

\begin{proof}
The proof is just a simple adaptation of Theorem \ref{T-2.6}. We can write for any $p>0$
$$
\Phi(A^{-p})^{-1/p}=\bigl\{\Phi((A^{-1})^p)^{1/p}\bigr\}^{-1},
$$
where $A^{-1}$ and $\{\cdots\}^{-1}$ are defined in the sense of the generalized inverse
so that
$$
A^{-1}=\sum_{k=1}^ma_k^{-1}P_k=\sum_{k=1}^ma_{m+1-k}^{-1}P_{m+1-k}
$$
with $a_m^{-1}>\dots>a_1^{-1}>0$. By Theorem \ref{T-2.6} we have
$$
\lim_{p\to\infty}\Phi((A^{-1})^p)^{1/p}
=\sum_{k=1}^ma_{m+1-k}^{-1}P_{\widetilde\cM_{m+1-k}}
=\sum_{k=1}^ma_k^{-1}P_{\widetilde\cM_k},
$$
where
\begin{align*}
\widetilde\cM_m&:=\ran\Phi(P_{m+1-1})=\ran\Phi(P_m), \\
\widetilde\cM_{m+1-k}
&:=\ran\Phi(P_{m+1-1}+\dots+P_{m+1-k})\ominus\ran\Phi(P_{m+1-1}+\dots+P_{m+1-(k-1)}) \\
&\ =\ran\Phi(P_{m+1-k}+\dots+P_m)\ominus\ran\Phi(P_{m+2-k}+\dots+P_m),
\quad2\le k\le m.
\end{align*}
According to the proofs of Theorems \ref{T-2.1} and \ref{T-2.4}, we see that the $i$th
eigenvalue $\lambda_i(p)$ of $\Phi((A^{-1})^p)^{1/p}$ converges to a positive real as
$p\to\infty$, or otherwise, $\lambda_i(p)=0$ for all $p>0$. That is, $\lambda_i(p)\to0$ as
$p\to\infty$ occurs only when $\lambda_i(p)=0$ for all $p>0$. This implies that
$$
\lim_{p\to\infty}\Phi(A^{-p})^{-1/p}
=\Bigl\{\lim_{p\to\infty}\Phi((A^{-1})^p)^{1/p}\Bigr\}^{-1}
=\sum_{k=1}^ma_kP_{\widetilde\cM_k}.
$$
\end{proof}

\begin{remark}\label{R-2.10}\rm
Assume that $\Phi:\bM_n\to\bM_{n'}$ is a \emph{unital} positive linear map. Let $A\in\bM_n$
be positive definite and $1\le p<q$. Since $x^{p/q}$ and $x^{1/p}$ are operator monotone
on $[0,\infty)$, we have $\Phi(A^q)^{p/q}\ge\Phi(A^p)$ and so
$\Phi(A^q)^{1/q}\ge\Phi(A^p)^{1/p}$. Hence $\Phi(A^p)^{1/p}$ increases as $1\le p\nearrow$.
Similarly, $\Phi(A^{-q})^{p/q}\ge\Phi(A^{-p})$ and so
$\Phi(A^{-q})^{-1/q}\le\Phi(A^{-p})^{-1/p}$ since $x^{-1/p}$ is operator monotone decreasing
on $(0,\infty)$. Hence $\Phi(A^{-p})^{-1/p}$ decreases as $1\le p\nearrow$. Moreover, since
$x^{-1}$ is operator convex on $(0,\infty)$, we have $\Phi(A^{-1})^{-1}\le\Phi(A)$. (See
\cite[Theorem 2.1]{AuHi1} for more details.) Combining altogether, when $A$ is positive
definite, we have
\begin{align}\label{F-2.9}
\Phi(A^{-p})^{-1/p}\le\Phi(A^q)^{1/q},\qquad p,q\ge1,
\end{align}
and in particular,
$$
\lim_{p\to\infty}\Phi(A^{-p})^{-1/p}\le\lim_{p\to\infty}\Phi(A^p)^{1/p}.
$$
However, the latter inequality does not hold unless $\Phi$ is unital. For example, let
$$
P_1:=\begin{bmatrix}1&0\\0&0\end{bmatrix},\quad
P_2:=\begin{bmatrix}0&0\\0&1\end{bmatrix},\quad
Q_1:=\begin{bmatrix}1/2&1/2\\1/2&1/2\end{bmatrix},\quad
Q_2:=\begin{bmatrix}1/2&-1/2\\-1/2&1/2\end{bmatrix},
$$
and consider $\Phi:\bM_2\to\bM_2$ given by
$$
\Phi\biggl(\begin{bmatrix}a_{11}&a_{12}\\a_{21}&a_{22}\end{bmatrix}\biggr)
:=a_{11}P_1+a_{22}Q_1,
$$
and $A:=aP_1+bP_2$ where $a>b>0$. Since $P_{\ran\Phi(P_1+P_2)}=I$, $P_{\ran\Phi(P_1)}=P_1$
and $P_{\ran\Phi(P_2)}=Q_1$, Theorems \ref{T-2.6} and \ref{T-2.9} give
$$
\lim_{p\to\infty}\Phi(A^p)^{1/p}=aP_1+b(I-P_1)=aP_1+bP_2,
$$
$$
\lim_{p\to\infty}\Phi(A^{-p})^{-1/p}=a(I-Q_1)+bQ_1=aQ_2+bQ_1.
$$
We compute
$$
(aP_1+bP_2)-(aQ_2+bQ_1)
=\begin{bmatrix}{a-b\over2}&{a-b\over2}\\{a-b\over2}&{b-a\over2}\end{bmatrix},
$$
which is not positive semidefinite.
\end{remark}

\begin{remark}\label{R-2.11}\rm
We may always assume that $\Phi:\bM_n\to\bM_{n'}$ is strictly positive. Indeed, we may
consider $\Phi$ as $\Phi:\bM_n\to Q_0\bM_{n'}Q_0\cong\bM_{n''}$, where $Q_0$ is the support
projection of $\Phi(I_n)$. Under this convention, another reasonable definition of
$\Phi(A^{-p})^{-1/p}$ for $p\ge1$ is
$$
\Phi(A^{-p})^{-1/p}:=\lim_{\eps\searrow0}\Phi((A+\eps I_n)^{-p})^{-1/p},
$$
which is well defined since $\Phi((A+\eps I)^{-p})$ is increasing so that
$\Phi((A+\eps I)^{-p})^{-1/p}$ is decreasing as $\eps\searrow0$. But this definition is
different from the above definition of $\Phi(A^{-p})^{-1/p}$. For example, let
$\Phi:\bM_2\to\bM_2$ be given by $\Phi(A):=KAK^*$ with an invertible
$K=\begin{bmatrix}a&b\\c&d\end{bmatrix}$, and let $A=P=\begin{bmatrix}1&0\\0&0\end{bmatrix}$.
Then $A^{-p}=P$ (in the generalized inverse) so that
$$
KA^{-p}K^*=\begin{bmatrix}|a|^2&a\overline c\\\overline ac&|c|^2\end{bmatrix}
$$
and so
\begin{align}\label{F-2.10}
(KA^{-p}K^*)^{-1/p}={1\over(|a|^2+|c|^2)^{1-{1\over p}}}
\begin{bmatrix}|a|^2&a\overline c\\\overline ac&|c|^2\end{bmatrix}.
\end{align}
On the other hand,
\begin{align*}
\lim_{\eps\searrow0}(K(A+\eps I)^{-p}K^*)^{-1/p}
&=\lim_{\eps\searrow0}(K^{*-1}(A+\eps I)^pK^{-1})^{1/p} \\
&=(K^{*-1}A^pK^{-1})^{1/p}=(K^{*-1}PK^{-1})^{1/p}
\end{align*}
is equal to
\begin{align}\label{F-2.11}
{1\over|ad-bc|^{2/p}(|b|^2+|d|^2)^{1-{1\over p}}}
\begin{bmatrix}|d|^2&-b\overline d\\-\overline bd&|b|^2\end{bmatrix}.
\end{align}
Hence we find that \eqref{F-2.10} and \eqref{F-2.11} are very different, even after taking
the limits as $p\to\infty$.

Here is a simpler example. Let $\ffi:\bM_2\to\bC=\bM_1$ be a state (hence, unital) with
density matrix $\begin{bmatrix}1/2&1/2\\1/2&1/2\end{bmatrix}$, and let
$A=\begin{bmatrix}1&0\\0&0\end{bmatrix}$. For the first definition we have
$$
\lim_{p\to\infty}\ffi(A^{-p})^{-1/p}=\lim_{p\to\infty}2^{1/p}=1.
$$
For the second definition,
$$
\lim_{\eps\searrow0}\ffi((A+\eps I)^{-p})^{-1/p}
=\lim_{\eps\searrow0}\biggl\{{(1+\eps)^{-p}+\eps^{-p}\over2}\biggr\}^{-1/p}=0
$$
for all $p>0$. Moreover, since $\ffi(A^p)^{1/p}=2^{-1/p}$ for $p>0$, this example says also
that \eqref{F-2.9} does not hold for general positive semidefinite $A$.
\end{remark}

\begin{problem}\label{Q-2.12}\rm
It is also interesting to consider the limit of $(A^pBA^p)^{1/p}$ as $p\to\infty$ for
$A,B\in\bM_n^+$, a version different from the limit treated in Theorem \ref{T-2.1}. To
consider $\lim_{p\to\infty}(A^pBA^p)^{1/p}$, we may assume without loss of generality that
$B$ is an orthogonal projection $E$ (see the argument around \eqref{F-3.2} below). Since
$(A^pEA^p)^{1/p}=(A^pE^{2p}A^p)^{1/p}$ converges as $p\to\infty$ by \cite[Theorem 2.5]{AuHi2},
the existence of the limit $\lim_{p\to\infty}(A^pBA^p)^{1/p}$ follows. But it seems that
the description of the limit is a combinatorial problem much more complicated than that in
Theorem \ref{T-2.1}.
\end{problem}

\section{$\lim_{p\to\infty}(A^p\sigma B)^{1/p}$ for operator means $\sigma$}

In theory of operator means due to Kubo and Ando \cite{KA}, a main result says that each
operator mean $\sigma$ is associated with a non-negative operator monotone function $f$ on
$[0,\infty)$ with $f(1)=1$ in such a way that
$$
A\sigma B=A^{1/2}f(A^{-1/2}BA^{-1/2})A^{1/2}
$$
for $A,B\in\bM_n^+$ with $A>0$, which is further extended to general $A,B\in\bM_n^+$ as
$$
A\sigma B=\lim_{\eps\searrow0}(A+\eps I)\sigma(B+\eps I).
$$
We write $\sigma_f$ for the operator mean associated with $f$ as above. For $0\le\alpha\le1$,
the operator mean corresponding to the function $x^\alpha$ ($x\ge0$) is the
\emph{weighted geometric mean} $\#_\alpha$, i.e.,
$$
A\#_\alpha B=A^{1/2}(A^{-1/2}BA^{-1/2})^\alpha A^{1/2}
$$
for $A,B\in\bM_n^+$ with $A>0$. In particular, $\#=\#_{1/2}$ is the so-called
\emph{geometric mean} first introduced by Pusz and Woronowicz \cite{PW}.

The transpose of $f$ above is given by
$$
\widetilde f(x):=xf(x^{-1}),\qquad x>0,
$$
which is again an operator monotone function on $[0,\infty)$ (after extending to $[0,\infty)$
by continuity) corresponding to the transposed operator mean of $\sigma_f$, i.e.,
$A\sigma_{\widetilde f}B=B\sigma_fA$. We also write
\begin{align}\label{F-3.1}
\widehat f(x):=\begin{cases}
\widetilde f(x^{-1})=f(x)/x & \text{if $x>0$}, \\
\quad0 & \text{if $x=0$}.
\end{cases}
\end{align}

In the rest of the section, let $f$ be such an operator monotone function as above and
$\sigma_f$ be the corresponding operator mean. We are concerned with the existence and the
description of the limit $\lim_{p\to\infty}(A^p\sigma_fB)^{1/p}$, in particular,
$\lim_{p\to\infty}(A^p\#_\alpha B)^{1/p}$ for $A,B\in\bM_n^+$. For this, we may assume
without loss of generality that $B$ is an orthogonal projection. Indeed, let $E$ be the
support projection of $B$. Then we can choose $\lambda,\mu>0$ with $\lambda<1<\mu$ such that
$\lambda E\le B\le\mu E$. Thanks to monotonicity and positive homogeneity of $\sigma_f$, we
have
$$
\lambda(A^p\sigma_f E)=(\lambda A^p)\sigma_f(\lambda E)
\le A^p\sigma_f B\le(\mu A^p)\sigma_f(\mu E)=\mu(A^p\,\sigma_f E).
$$
Hence, for every $p\ge1$, since $x^{1/p}$ ($x\ge0$) is operator monotone,
\begin{align}\label{F-3.2}
\lambda^{1/p}(A^p\sigma_f E)^{1/p}\le(A^p\sigma_f B)^{1/p}\le\mu^{1/p}(A^p\sigma_f B),
\end{align}
so that $\lim_{p\to\infty}(A^p\sigma_f B)^{1/p}$ exists if and only if
$\lim_{p\to\infty}(A^p\sigma_f E)^{1/p}$ does, and in this case, both limits are equal.
In particular, when $B>0$, since $(A^p\sigma_f I)^{1/p}=\widetilde f(A^p)^{1/p}$, we note that
$$
\lim_{p\to\infty}(A^p\sigma_f B)^{1/p}=\widetilde f^{(\infty)}(A)
$$
whenever $\widetilde f^{(\infty)}(x):=\lim_{p\to\infty}\widetilde f(x^p)^{1/p}$ exists for
all $x\ge0$. For instance,
\begin{itemize}
\item if $f(x)=1-\alpha+\alpha x$ where $0<\alpha<1$, then $\sigma_f=\triangledown_\alpha$,
the $\alpha$-arithmetic mean $A\triangledown_\alpha B:=(1-\alpha)A+\alpha B$, and
$\widetilde f^{(\infty)}(x)=\max\{x,1\}$,
\item if $f(x)=x^\alpha$ where $0\le\alpha\le1$, then $\sigma_f=\#_\alpha$ and
$\widetilde f^{(\infty)}(x)=\widetilde f(x)=x^{1-\alpha}$,
\item if $f(x)=x/((1-\alpha)x+\alpha)$ where $0<\alpha<1$, then $\sigma_f=\,!_\alpha$, the
$\alpha$-harmonic mean $A\,!_\alpha B:=(A^{-1}\triangledown_\alpha B)^{-1}$, and
$\widetilde f^{(\infty)}(x)=\min\{x,1\}$.
\end{itemize}
But it is unknown to us that, for any operator monotone function $f$ on $[0,\infty)$, the
limit $\lim_{p\to\infty}f(x^p)^{1/p}$ exists for all $x\ge0$, while it seems so.

When $E$ is an orthogonal projection, the next proposition gives a nice expression for
$A\sigma_fE$. This was shown in \cite[Lemma 4.7]{Hi2}, while the proof is given here for the
convenience of the reader.

\begin{lemma}\label{L-3.1}
Assume that $f(0)=0$. If $A\in\bM_n$ is positive definite and $E\in\bM_n$ is an orthogonal
projection, then
\begin{align}\label{F-3.3}
A\sigma_fE=\widehat f(EA^{-1}E),
\end{align}
where $\widehat f$ is given in \eqref{F-3.1}.
\end{lemma}

\begin{proof}
For every $m=1,2,\dots$ we have
\begin{align}\label{F-3.4}
A^{-1/2}(EA^{-1}E)^mA^{-1/2}=(A^{-1/2}EA^{-1/2})^{m+1}.
\end{align}
Note that the eigenvalues of $EA^{-1}E$ and those of $A^{-1/2}EA^{-1/2}$ are the same
including multiplicities. Choose a $\delta>0$ such that the positive eigenvalues of
$EA^{-1}E$ and $A^{-1/2}EA^{-1/2}$ are included in $[\delta,\delta^{-1}]$. Then, since
$\widehat f(x)$ is continuous on $[\delta,\delta^{-1}]$, one can choose a sequence of
polynomials $p_k(x)$ with $p_k(0)=0$ such that $p_k(x)\to\widehat f(x)$ uniformly on
$[\delta,\delta^{-1}]$ as $n\to\infty$. By \eqref{F-3.4} we have
$$
A^{-1/2}p_k(EA^{-1}E)A^{-1/2}=A^{-1/2}EA^{-1/2}p_k(A^{-1/2}EA^{-1/2})
$$
for every $k$. Since $\widehat f(0)=0$ by definition, we have
$$
p_k(EA^{-1}E)\longrightarrow\widehat f(EA^{-1}E)
$$
and
$$
A^{-1/2}EA^{-1/2}p_k(A^{-1/2}EA^{-1/2})\longrightarrow
A^{-1/2}EA^{-1/2}\widehat f(A^{-1/2}EA^{-1/2})
$$
as $k\to\infty$. Since $f(0)=0$ by assumption, we have $f(x)=x\widehat f(x)$ for all
$x\in[0,\infty)$. This implies that
$$
A^{-1/2}EA^{-1/2}\widehat f(A^{-1/2}EA^{-1/2})=f(A^{-1/2}EA^{-1/2}).
$$
Therefore,
$$
A^{-1/2}\widehat f(EA^{-1}E)A^{-1/2}=f(A^{-1/2}EA^{-1/2})
$$
so that we have $\widehat f(EA^{-1}E)=A^{1/2}f(A^{-1/2}EA^{-1/2})A^{1/2}=A\sigma_fE$, as
asserted.
\end{proof}

Formula \eqref{F-3.3} can equivalently be written as
\begin{align}\label{F-3.5}
A\sigma_fE=\widetilde f((EA^{-1}E)^{-1}),
\end{align}
where $(EA^{-1}E)^{-1}$ is the inverse restricted to $\ran E$ (in the sense of the generalized
inverse) and $\widetilde f((EA^{-1}E)^{-1})$ is also restricted to $\ran E$. In particular, if
$f$ is symmetric (i.e., $f=\widetilde f$) with $f(0)=0$, then
$$
A\sigma_fE=f((EA^{-1}E)^{-1}).
$$

\begin{example}\label{E-3.2}\rm
Assume that $0<\alpha\le1$ and $A,E$ are as in Lemma \ref{L-3.1}.
\begin{itemize}
\item[(1)] When $f(x)=x^\alpha$ and $\sigma_f=\#_\alpha$, $\widehat f(x)=x^{\alpha-1}$ for
$x>0$ so that
$$
A\#_\alpha E=(EA^{-1}E)^{\alpha-1},
$$
where the ($\alpha-1$)-power in the right-hand side is defined with restriction to $\ran E$.

\item[(2)] When $f(x)=x/((1-\alpha)x+\alpha)$ and $\sigma_f=\,!_\alpha$,
$\widehat f(x)=(1-\alpha+\alpha x)^{-1}$ for $x>0$ so that
$$
A\,!_\alpha E=\{(1-\alpha)E+\alpha EA^{-1}E\}^{-1}=\{E((1-\alpha)I+\alpha A^{-1})E\}^{-1},
$$
where the inverse of $E((1-\alpha)I+\alpha A^{-1})E$ in the right-hand side is restricted to
$\ran E$.

\item[(3)] When $f(x)={(x-1)/\log x}$ and so $\sigma_f$ is the logarithmic mean,
$\widehat f(x)=(1-x^{-1})/\log x$ for $x>0$ so that
$$
A\sigma_f E=(E-(EA^{-1}E)^{-1})(\log EA^{-1}E)^{-1},
$$
where the right-hand side is defined with restriction to $\ran E$.
\end{itemize}
\end{example}

\begin{thm}\label{T-3.3}
Assume that $f(0)=0$ and $f(x^r)\ge f(x)^r$ for all $x>0$ and all $r\in(0,1)$. Let
$A\in\bM_n^+$ and $E\in\bM_n$ be an orthogonal projection. Then
\begin{align}\label{F-3.6}
(A^p\,\sigma_f\,E)^{1/p}\ge(A^q\,\sigma_f\,E)^{1/q}\quad\mbox{if $1\le p<q$}.
\end{align}
\end{thm}

\begin{proof}
First, note that $\widetilde f(x^r)=x^rf(x^{-r})\ge x^rf(x^{-1})^r=\widetilde f(x)^r$ for
all $x>0$, $r\in(0,1)$. By replacing $A$ with $A+\eps I$ and taking the limit as
$\eps\searrow0$, we may assume that $A$ is positive definite. Let $1\le p<q$ and
$r:=p/q\in(0,1)$. By \eqref{F-3.5} we have
\begin{align}\label{F-3.7}
(A^q\,\sigma_f\,E)^r=\widetilde f((EA^{-q}E)^{-1})^r
\le\widetilde f((EA^{-q}E)^{-r}).
\end{align}
Since $x^r$ is operator monotone on $[0,\infty)$, we have by Hansen's inequality \cite{Ha}
$$
(EA^{-q}E)^r\ge EA^{-qr}E=EA^{-p}E
$$
so that $(EA^{-q}E)^{-r}\le(EA^{-p}E)^{-1}$. Since $\widetilde f(x)$ is operator monotone on
$[0,\infty)$, we have
\begin{align}\label{F-3.8}
\widetilde f((EA^{-q}E)^{-r})\le\widetilde f((EA^{-p}E)^{-1})=A^p\,\sigma_f\,E.
\end{align}
Combining \eqref{F-3.7} and \eqref{F-3.8} gives
$$
(A^q\,\sigma_f\,E)^r\le A^p\,\sigma_f\,E.
$$
Since $x^{1/p}$ is operator monotone on $[0,\infty)$, we finally have
$$
(A^q\,\sigma_f\,E)^{1/q}\le(A^p\,\sigma_f\,E)^{1/p}.
$$
\end{proof}

\begin{cor}\label{C-3.4}
Assume that $f(0)=0$ and $f(x^r)\ge f(x)^r$ for all $x>0$, $r\in(0,1)$. Then for every
$A,B\in\bM_n^+$, the limit
$$
\lim_{p\to\infty}(A^p\,\sigma_f\,B)^{1/p}
$$
exists.
\end{cor}

\begin{proof}
From the argument around \eqref{F-3.2} we may assume that $B$ is an orthogonal projection
$E$. Then Theorem \ref{T-3.3} implies that $(A^p\,\sigma_f\,E)^{1/p}$ converges as
$p\to\infty$.
\end{proof}

\begin{remark}\label{R-3.5}\rm
Following \cite{Wa}, an operator monotone function $f$ on $[0,\infty)$ is said to be
\emph{power monotone increasing} (p.m.i.\ for short) if $f(x^r)\ge f(x)^r$ for all $x>0$,
$r>1$ (equivalently, $f(x^r)\le f(x^r)$ for all $x>0$, $r\in(0,1)$), and
\emph{power monotone decreasing} (p.m.d.) if $f(x^r)\le f(x)^r$ for all $x>0$, $r>1$.
These conditions play a role to characterize the operator means $\sigma_f$ satisfying
Ando-Hiai's inequality \cite{AH}, see \cite[Lemmas 2.1, 2.2]{Wa}. For instance,
the p.m.d.\ condition is satisfied for $f$ in (1) and (2) of Example \ref{E-3.2}, while $f$
in Example \ref{E-3.2}\,(3) does the p.m.i.\ condition. Hence, for any $\alpha\in[0,1]$,
$(A^p\#_\alpha\,E)^{1/p}$ and $(A^p\,!_\alpha E)^{1/p}$ converge decreasingly as
$1\le p\nearrow\infty$. In fact, for the harmonic mean, we have the limit
$A\wedge B:=\lim_{p\to\infty}(A^p\,!\,B^p)^{1/p}$, the  decreasing limit as
$1\le p\nearrow\infty$ for any $A,B\ge0$, which is the infimum counterpart of $A\vee B$ in
\cite{Ka} (see also Example \ref{E-2.7}). The reader might be wondering if the opposite
inequality to \eqref{F-3.6} holds (i.e., $(A^p\,\sigma_f\,E)^{1/p}$ is increasing as
$1\le p\nearrow\infty$) when $f$ satisfies the p.m.i.\ condition. Although this is the case
when $\sigma=\triangledown_\alpha$ the weighted arithmetic mean, it is not the case in
general. In fact, if it were true, $(A^p\#_\alpha E)^{1/p}$ must be constant for $p\ge1$
since $x^\alpha$ satisfies both p.m.i.\ and p.m.d.\ conditions, that is impossible.
\end{remark}

Finally, for the weighted geometric mean $\#_\alpha$ we obtain the explicit description of
$\lim_{p\to\infty}(A^p\#_\alpha E)^{1/p}$ for any $A\in\bM_n^+$. For the trivial cases
$\alpha=0,1$ note that $(A^p\#_0E)^{1/p}=A$ and $(A^p\#_1E)^{1/p}=E$ for all $p>0$.

\begin{thm}\label{T-3.6}
Assume that $0<\alpha<1$. Let $A\in\bM_n^+$ be given with the spectral decomposition
$A=\sum_{k=1}^ma_kP_k$ where $a_1>\dots>a_m>0$, and $E\in\bM_n$ be an orthogonal projection.
Then
\begin{align}\label{F-3.9}
\lim_{p\to\infty}(A^p\#_\alpha E)^{1/p}
=\sum_{k=1}^ma_k^{1-\alpha}Q_k,
\end{align}
where
\begin{align*}
Q_1&:=P_1\wedge E, \\
Q_k&:=(P_1+\dots+P_k)\wedge E-(P_1+\dots+P_{k-1})\wedge E,\qquad2\le k\le m.
\end{align*}
\end{thm}

\begin{proof}
First, assume that $A$ is positive definite so that $P_1+\dots+P_m=I$.
When $f(x)=x^\alpha$ with $0<\alpha<1$, formula \eqref{F-3.5} is given as
$$
A\#_\alpha E=(EA^{-1}E)^{-(1-\alpha)}.
$$
Since
$$
\lim_{p\to\infty}(A^p\,\#_\alpha\,E)^{1/p}
=\lim_{p\to\infty}(EA^{-p}E)^{-(1-\alpha)/p}
=\lim_{p\to\infty}(E(A^{1-\alpha})^{-p}E)^{-1/p},
$$
it follows from Theorem \ref{T-2.9} that
$$
\lim_{p\to\infty}(A^p\#_\alpha E)^{1/p}
=\sum_{k=1}^ma_k^{1-\alpha}P_{\widetilde\cM_k},
$$
where
\begin{align*}
\widetilde\cM_k&:=\ran E(P_k+\dots+P_m)E\ominus\ran E(P_{k+1}+\dots+P_m)E,
\quad1\le k\le m-1, \\
\widetilde\cM_m&:=\ran EP_mE.
\end{align*}
From Lemma \ref{L-3.7} below we have
$$
\widetilde\cM_1=\ran E\ominus\ran EP_1^\perp E=\ran P_1\wedge E,
$$
and for $2\le k\le m$,
\begin{align*}
\widetilde\cM_k&=\ran E(P_1+\dots+P_{k-1})^\perp E
\ominus\ran E(P_1+\dots+P_k)^\perp E \\
&=\bigl[\ran E\ominus\ran E(P_1+\dots+P_k)^\perp E\bigr]
\ominus\bigl[\ran E\ominus\ran E(P_1+\dots+P_{k-1})^\perp E\bigr] \\
&=\ran(P_1+\dots+P_k)\wedge E\ominus\ran(P_1+\dots+P_{k-1})\wedge E \\
&=\ran\bigl[(P_1+\dots+P_k)\wedge E-(P_1+\dots+P_{k-1})\wedge E\bigr].
\end{align*}
Therefore, \eqref{F-3.9} is established when $A$ is positive definite.

Next, when $A$ is not positive definite, let $P_{m+1}:=(P_1+\dots+P_m)^\perp$. For any
$\eps\in(0,a_m)$ define $A_\eps:=A+\eps P_{m+1}$. Then the above case implies that
$$
\lim_{p\to\infty}(A_\eps\#_\alpha E)^{1/p}
=\sum_{k=1}^ma_k^{1-\alpha}Q+\eps^{1-\alpha}Q_{m+1},
$$
where
$$
Q_{m+1}:=E-(P_1+\dots+P_m)\wedge E.
$$
Assume that $0<\eps<\eps'<a_m$. For every $p\ge1$, since $A_\eps^p\le A_{\eps'}^p$, we have
$A_\eps^p\#_\alpha E\le A_{\eps'}^p\#_\alpha E$ and hence
$(A_\eps^p\#_\alpha E)^{1/p}\le(A_{\eps'}^p\#_\alpha E)^{1/p}$. Furthermore, since
$A_\eps^p\#_\alpha E\to A^p\#_\alpha E$ as $a_m>\eps\searrow0$, we have
\begin{align}\label{F-3.10}
(A^p\#_\alpha E)^{1/p}=\lim_{a_m>\eps\searrow0}(A_\eps^p\#_\alpha E)^{1/p}
\quad\mbox{decreasingly}.
\end{align}
Now, we can perform a calculation of limits as follows:
\begin{align*}
\lim_{1\le p\to\infty}(A^p\#_\alpha E)^{1/p}
&=\lim_{1\le p\to\infty}\,\lim_{a_m>\eps\searrow0}(A_\eps^p\#_\alpha E)^{1/p} \\
&=\lim_{a_m>\eps\searrow0}\,\lim_{1\le p\to\infty}(A_\eps^p\#_\alpha E)^{1/p} \\
&=\lim_{a_m>\eps\searrow0}\Biggl(\sum_{k=1}^ma_k^{1-\alpha}Q+\eps^{1-\alpha}Q_{m+1}\Biggr) \\
&=\sum_{k=1}^ma_k^{1-\alpha}Q_k.
\end{align*}
In the above, the second equality (the exchange of two limits) is confirmed as follows.
Let $X_{p,\eps}:=(A_\eps^p\#_\alpha E)^{1/p}$ for $p\ge1$ and $0<\eps<a_m$. Then $X_{p,\eps}$
is decreasing as $1\le p\nearrow\infty$ by Theorem \ref{F-3.3} and also decreasing as
$a_m>\eps\searrow0$ as seen in \eqref{F-3.10}. Let $X_p:=\lim_\eps X_{p,\eps}$
($=(A^p\#_\alpha E)^{1/p}$), $X_\eps:=\lim_pX_{p,\eps}$, and $X:=\lim_pX_p$. Since
$X_{p,\eps}\ge X_p$, we have $X_\eps\ge X$ and hence $\lim_\eps X_\eps\ge X$. On the other
hand, since $X_{p,\eps}\ge X_\eps$, we have $X_p\ge\lim_\eps X_\eps$ and hence
$X\ge\lim_\eps X_\eps$. Therefore, $X=\lim_\eps X_\eps$, which gives the assertion.
\end{proof}

In particular, when $A=P$ is an orthogonal projection, we have
$(A^p\#_\alpha E)^{1/p}=P\wedge E$ for all $p>0$ (see \cite[Theorem 3.7]{KA}) so that both
sides of \eqref{F-3.9} are certainly equal to $P\wedge E$.

\begin{lemma}\label{L-3.7}
For every orthogonal projections $E$ and $P$,
$$
\ran EP^\perp E=\ran(E-P\wedge E),
$$
or equivalently,
$$
\ran P\wedge E=\ran E\ominus\ran EP^\perp E.
$$
\end{lemma}

\begin{proof}
According to the well-known representation of two projections (see \cite[pp.~306--308]{Ta}),
we write
\begin{align*}
E&=I\oplus I\oplus0\oplus\begin{bmatrix}I&0\\0&0\end{bmatrix}\oplus0, \\
P&=I\oplus0\oplus I\oplus\begin{bmatrix}C^2&CS\\CS&S^2\end{bmatrix}\oplus0,
\end{align*}
where $0<C,S<I$ with $C^2+S^2=I$. We have
$$
P\wedge E=I\oplus0\oplus0\oplus0\oplus0.
$$
Since
$$
P^\perp=0\oplus I\oplus0\oplus\begin{bmatrix}S^2&-CS\\-CS&C^2\end{bmatrix}\oplus I,
$$
we also have
$$
EP^\perp E=0\oplus I\oplus0\oplus\begin{bmatrix}S^2&0\\0&0\end{bmatrix}\oplus0,
$$
whose range is that of
$$
0\oplus I\oplus0\oplus\begin{bmatrix}I&0\\0&0\end{bmatrix}\oplus0=E-P\wedge E,
$$
which yields the conclusion.
\end{proof}

\subsection*{Acknowledgments}

This work was supported by JSPS KAKENHI Grant Number JP17K05266.


\begin{thebibliography}{99}

\bibitem{AH}
T. Ando and F. Hiai, Log majorization and complementary Golden-Thompson type inequalities,
{\it Linear Algebra Appl.} {\bf 197/198} (1994), 113--131.

\bibitem{ACS1}
J. Antezana, G. Corach and D. Stojanoff, Spectral shorted matrices,
{\it Linear Algebra Appl.} {\bf 381} (2004), 197--217.

\bibitem{ACS2}
J. Antezana, G. Corach and D. Stojanoff, Spectral shorted operators,
{\it Integral Equations Operator Theory} {\bf 55} (2006), 169--188.

\bibitem{Ar}
H. Araki, On an inequality of Lieb and Thirring, {\it Lett. Math. Phys.} {\bf 19}
(1990), 167--170.

\bibitem{AuHi1}
K. M. R. Audenaert and F. Hiai, On matrix inequalities between the power means:
counterexamples, {\it Linear Algebra Appl.} {\bf 439} (2013), 1590--1604.

\bibitem{AuHi2}
K. M. R. Audenaert and F. Hiai, Reciprocal Lie-Trotter formula,
{\it Linear and Multilinear Algebra} {\bf 64} (2016), 1220--1235.

\bibitem{Bo}
J.-C. Bourin, Convexity or concavity inequalities for Hermitian operators,
{\it Math. Ineq. Appl.} {\bf 7} (2004), 607--620.

\bibitem{DL}
N. Datta and F. Leditzky, A limit of the quantum R\'enyi divergence,
{\it J. Phys. A: Math. Theor.} {\bf 47} (2014), 045304.

\bibitem{Ha}
F. Hansen, An operator inequality, {\it Math. Ann.} {\bf 246} (1980), 249--250.

\bibitem{Hi}
F. Hiai, Equality cases in matrix norm inequalities of Golden-Thompson type,
{\it Linear and Multilinear Algebra} {\bf 36} (1994), 239--249.

\bibitem{Hi2}
F. Hiai, A generalization of Araki's log-majorization,
{\it Linear Algebra Appl.} {\bf 501} (2016), 1--16.

\bibitem{Ka}
T. Kato, Spectral order and a matrix limit theorem,
{\it Linear and Multilinear Algebra} {\bf 8} (1979), 15--19.

\bibitem{KA}
F. Kubo and T. Ando, Means of positive linear operators,
{\it Math. Ann.} {\bf 246} (1980), 205--224.

\bibitem{MDSFT}
M. M\"uller-Lennert, F. Dupuis, O. Szehr, S. Fehr and M. Tomamichel,
On quantum R\'enyi entropies: A new generalization and some properties,
{\it J. Math. Phys.} {\bf 54} (2013), 122203.

\bibitem{PW}
W. Pusz and S. L. Woronowicz, Functional calculus for sesquilinear forms and the
purification map, {\it Rep. Math. Phys.} {\bf 8} (1975), 159--170.

\bibitem{Ta}
M. Takesaki, {\it Theory of Operator Algebras I}, Encyclopaedia of Mathematical
Sciences, Vol. 124, Springer-Verlag, Berlin, 2002.

\bibitem{Wa}
S. Wada, Some ways of constructing Furuta-type inequalities,
{\it Linear Algebra Appl.} {\bf 457} (2014), 276--286.

\bibitem{WWY}
M. M. Wilde, A. Winter and D. Yang,
Strong converse for the classical capacity of entanglement-breaking
and Hadamard channels via a sandwiched R\'enyi relative Entropy,
{\it Comm. Math. Phys.} {\bf 331} (2014), 593--622.

\end{thebibliography}
\end{document}